\documentclass[12pt]{article}
\usepackage{fullpage,graphicx,psfrag,amsmath,amsfonts,verbatim,url,amssymb,amsthm}
\usepackage[small,bf]{caption}
\usepackage[font=footnotesize]{subcaption}
\usepackage{algorithm}
\usepackage[noend]{algpseudocode}
\usepackage[T1]{fontenc}
\usepackage{tikz}
\usetikzlibrary{positioning,calc}
\definecolor {processblue}{cmyk}{0.96,0,0,0}
\theoremstyle{definition}

\algrenewcommand\alglinenumber[1]{
   {\sf\footnotesize\addfontfeatures{Colour=888888,Numbers=Monospaced}#1}}
\algrenewcommand\algorithmicrequire{\textbf{Precondition:}}
\algrenewcommand\algorithmicensure{\textbf{Postcondition:}}
\mathchardef\mhyphen="2D
\renewcommand{\algorithmicrequire}{\textbf{Input:}}
\renewcommand{\algorithmicensure}{\textbf{Output:}}

\newcommand{\BEAS}{\begin{eqnarray*}}
\newcommand{\EEAS}{\end{eqnarray*}}
\newcommand{\BEA}{\begin{eqnarray}}
\newcommand{\EEA}{\end{eqnarray}}
\newcommand{\BEQ}{\begin{equation}}
\newcommand{\EEQ}{\end{equation}}
\newcommand{\BIT}{\begin{itemize}}
\newcommand{\EIT}{\end{itemize}}
\newcommand{\BNUM}{\begin{enumerate}}
\newcommand{\ENUM}{\end{enumerate}}

\newcommand{\BA}{\begin{array}}
\newcommand{\EA}{\end{array}}

\newcommand{\eg}{{\it e.g.}}
\newcommand{\ie}{{\it i.e.}}

\newcommand{\ones}{\mathbf 1}

\newcommand{\reals}{{\mbox{\bf R}}}

\newcommand{\diag}{\mathop{\bf diag}}

\newcommand{\K}{\mathcal{K}}

\newtheorem{theorem}{Theorem}[section]

{\begin{quote}}{\end{quote}}

\makeatletter
\long\def\@makecaption#1#2{
   \vskip 9pt
   \begin{small}
   \setbox\@tempboxa\hbox{{\bf #1:} #2}
   \ifdim \wd\@tempboxa > 5.5in
        \begin{center}
        \begin{minipage}[t]{5.5in}
        \addtolength{\baselineskip}{-0.95pt}
        {\bf #1:} #2 \par
        \addtolength{\baselineskip}{0.95pt}
        \end{minipage}
        \end{center}
   \else
    \hbox to\hsize{\hfil\box\@tempboxa\hfil}
   \fi
   \end{small}\par
}
\makeatother

\newcounter{oursection}

\newcounter{lecture}

\title{Stochastic Matrix-Free Equilibration}
\author{Steven Diamond \and Stephen Boyd}

\begin{document}
\maketitle

\begin{abstract}
We present a novel method for approximately equilibrating
a matrix $A \in \reals^{m \times n}$
using only multiplication by $A$ and $A^T$.
Our method is based on convex optimization and projected stochastic
gradient descent, using an unbiased estimate of a gradient
obtained by a randomized method.
Our method provably converges in expectation with an $O(1/t)$
convergence rate and empirically gets good results with a small
number of iterations.
We show how the method can be applied as a preconditioner for
matrix-free iterative algorithms such as LSQR and Chambolle-Cremers-Pock,
substantially reducing the iterations required to reach a given level
of precision.
We also derive a novel connection between equilibration and
condition number, showing that equilibration minimizes an upper bound
on the condition number over all choices of row and column
scalings.
\end{abstract}

\section{Equilibration}
Equilibration refers to scaling the rows and columns of a matrix so the
norms of the rows are the same, and the norms of the columns are the same.
Given a matrix $A \in \reals^{m \times n}$, the goal is to find
diagonal matrices $D \in \reals^{m \times m}$ and $E \in \reals^{n \times n}$
so that the rows of $DAE$ all have $\ell_p$-norm $\alpha$ and the
columns of $DAE$ all have $\ell_p$-norm $\beta$.
(The row and column norm values $\alpha$ and $\beta$
are related by $m\alpha^p = n\beta^p$ for $p < \infty$.)
Common choices of $p$ are $1$, $2$, and $\infty$; in this paper,
we will focus on $\ell_2$-norm equilibration.
Without loss of generality, we assume throughout
that the entries of $D$ and $E$ are nonnegative.

Equilibration has applications to a variety of problems,
including target tracking in sensor networks \cite{HBRSGT:03},
web page ranking \cite{K:08}, and adjusting contingency tables
to match known marginal probabilities \cite{SZ:90}.
The primary use of equilibration, however, is as a heuristic method for
reducing condition number \cite{bradley2010algorithms};
in turn, reducing condition number is a heuristic for speeding
up a variety of iterative algorithms \cite[Chap.~5]{NW:06},
\cite{TJ:14,giselsson2014diagonal}.
Using equilibration to accelerate iterative algorithms is connected to
the broader notion of diagonal preconditioning,
which has been a subject of research for decades;
see, \eg, \cite[Chap.~2]{K:95}, \cite[Chap.~10]{G:97}, \cite{pock2011diagonal,GB:15}.

Equilibration has several justifications as a heuristic for reducing
condition number.
We will show in \S\ref{sec:equil_cond} that if $A$ is square and nonsingular,
any $D$ and $E$ that equilibrate $A$ in the $\ell_2$-norm
minimize a tight upper bound on $\kappa(DAE)$ over all diagonal $D$ and $E$.
Scaling only the rows or only the columns of $A$ so they have the same
$\ell_2$-norms (rather than scaling both at once, as we do here)
also has a connection
with minimizing condition number \cite{sluis1969condition}.

Another perspective is that equilibration minimizes a lower bound
on the condition number.
It is straightforward to show that the ratio between the largest and
smallest $\ell_2$-norms of the columns of $DAE$ is a lower bound on $\kappa(DAE)$:
\begin{align*}
\kappa(DAE) &= \frac{\sup_{\|x\|_2 = 1}\|DAEx\|_2}{\inf_{\|x\|_2 = 1}\|DAEx\|_2} \\
&\geq \frac{\max_{x \in \{e_1,\ldots,e_n \}}\|DAEx\|_2}
{\min_{x \in \{e_1,\ldots,e_n \}}\|DAEx\|_2}.
\end{align*}
The same inequality holds for the ratio between largest and
smallest $\ell_2$-norms of the rows of $DAE$.
For an equilibrated matrix these ratios are one, the smallest they can be.

Equilibration is an old problem and many techniques have been developed for
it, such as the Sinkhorn-Knopp \cite{sinkhorn1967concerning} and
Ruiz algorithms \cite{ruiz2001scaling}.
Existing $\ell_p$-norm equilibration methods
require knowledge of (the entries of) the matrix $|A|^p$, where the function
$|\cdot|^p$ is applied elementwise.
For this reason equilibration cannot be used in matrix-free methods,
which only interact with the matrix $A$ via multiplication of a vector by $A$ or by $A^T$.
Such matrix-free methods play a crucial role in scientific computing
and optimization.
Examples include the conjugate gradients method \cite{HeS:52},
LSQR \cite{PaS:82}, and the Chambolle-Cremers-Pock algorithm \cite{CP:11}.

In this paper we introduce a stochastic matrix-free equilibration
method that provably converges in expectation to the correct $D$ and $E$.
Our method builds on work by Bradley \cite{bradley2010algorithms},
who proposed a matrix-free equilibration algorithm with promising empirical
results but no theoretical guarantees.
Examples demonstrate that our matrix-free equlibration method converges
far more quickly than the theoretical analysis suggests, delivering
effective equilibration in a few tens of iterations, each involving
one multiplication by $A$ and one by $A^T$.
We demonstrate the method on examples of matrix-free iterative algorithms.
We observe that the cost of equilibration
is more than compensated for by the speedup of the iterative algorithm due to
diagonal preconditioning.
We show how our method can be modified to handle variants of the
equilibration problem, such as symmetric and block equilibration.

\section{Equilibration via convex optimization}\label{sec:equil-opt}

\subsection{The equilibration problem}

Equilibration can be posed as the convex
optimization problem \cite{BoV:04}
\begin{equation}\label{prob:equil}
\begin{array}{ll}
\mbox{minimize} & (1/2)
\sum_{i=1}^m\sum_{j=1}^{n} (A_{ij})^2e^{2u_i + 2v_j} - \alpha^2 \ones^Tu - \beta^2 \ones^Tv,
\end{array}
\end{equation}
where $u \in \reals^m$ and $v \in \reals^n$ are the optimization variables \cite{BHT:04}.
The diagonal matrices $D$ and $E$ are obtained via
\[
D = \diag(e^{u_1},\ldots, e^{u_m}), \qquad E = \diag(e^{v_1},\ldots, e^{v_n}).
\]
The optimality conditions for problem (\ref{prob:equil}) are precisely
that $DAE$ is equilibrated, \ie,
\[
|DAE|^2\ones = \alpha^2 \ones, \qquad |EA^TD|^2\ones = \beta^2\ones.
\]
The problem (\ref{prob:equil}) is unbounded below precisely when the matrix $A$
cannot be equilibrated.

Problem (\ref{prob:equil}) can be solved using a variety of methods for
smooth convex optimization \cite{BoV:04,NW:06}.
One attractive method, which exploits the special structure of the
objective, is to alternately minimize over $u$ and $v$.
We minimize over $u$ (or equivalently $D$)  by setting
\[
D_{ii} = \alpha \left(\sum_{j=1}^nA_{ij}^2E_{jj}^2 \right)^{-1/2}, \quad i=1,\ldots,m.
\]
We minimize over $v$ ($E$) by setting
\[
E_{jj} = \beta \left(\sum_{i=1}^mA_{ij}^2D_{ii}^2 \right)^{-1/2}, \quad j=1,\ldots,n.
\]
When $m = n$ and $\alpha = \beta = 1$, the above updates are precisely
the Sinkhorn-Knopp algorithm.
In other words, the Sinkhorn-Knopp algorithm is alternating block minimization
for the problem \eqref{prob:equil}.

\subsection{Equilibration and condition number}\label{sec:equil_cond}
In this subsection we show that equilibrating a square matrix minimizes
an upper bound on the condition number.
We will not use these results in the
sequel, where we focus on matrix-free methods for equilibration.

For $U\in \reals^{n \times n}$ nonsingular define the function $\Phi$ by
\[
\Phi(U) =
\exp\left(\|U\|_F^2 /2 \right)/\det \left(U^TU)^{1/2} \right)
= \exp\left( \sum_{i=1}^n \sigma_i^2/2 \right)/\prod_{i=1}^n \sigma_i,
\]
where $\sigma_1 \geq  \cdots \geq \sigma_n > 0$ are the singular values of $U$.
(Here $\| U \|_F$ denotes the Frobenius norm.)

\begin{theorem}\label{thm:sing-vals}
Let $A$ be square and invertible.
Then diagonal $D$ and $E$ equilibrate $A$, with row and column norms one,
if and only if they minimize $\Phi(DAE)$ over $D$ and $E$ diagonal.
\end{theorem}

\begin{proof}
We first rewrite problem (\ref{prob:equil}) in terms of $D$ and $E$ to obtain
\[
\begin{array}{ll}
\mbox{minimize} & (1/2)\|DAE\|_{F}^2 - \sum_{i=1}^n \log D_{ii}
-  \sum_{j=1}^n \log E_{jj} \\
\mbox{subject to} & \diag(D) > 0, \quad \diag(E) > 0, \quad
D,E \text{ diagonal},
\end{array}
\]
(Here we take $\alpha = \beta =1$, so the row and column norms are one.)
We can rewrite this problem as
\[
\begin{array}{ll}
\mbox{minimize} & (1/2)\|DAE\|_{F}^2 - \log \det\left((DAE)^T(DAE)\right)^{1/2} \\
\mbox{subject to} & \diag(D) > 0, \quad \diag(E) > 0, \quad
D, E \text{ diagonal},
\end{array}
\]
since the objective differs from the problem above
only by the constant $\log \det(A^TA)^{1/2}$.
Finally, taking the exponential of the objective,
we obtain the equivalent problem
\[
\begin{array}{ll}
\mbox{minimize} & \Phi(DAE) =\exp\left((1/2)\|DAE\|_{F}^2 \right)/
\det\left((DAE)^T(DAE)\right)^{1/2} \\
\mbox{subject to} & \diag(D) > 0, \quad \diag(E) > 0, \quad
D, E \text{ diagonal}.
\end{array}
\]
Thus diagonal (positive) $D$ and $E$ equilibrate $A$, with row and column
norms one, if and only if they minimize the objective of this problem.
\end{proof}

Theorem \ref{thm:sing-vals} links equilibration
with minimization of condition number because the function
$\Phi(DAE)$ gives an upper bound on $\kappa(DAE)$.

\begin{theorem}\label{thm:cond-ineq}
Let $U \in \reals^{n \times n}$ be nonsingular with singular values
$\sigma_1 \geq  \cdots \geq \sigma_n >0$ and condition number
$\kappa = \sigma_1/\sigma_n$. Then
\begin{equation}\label{ineq:cond}
2e^{-n/2} \Phi(U) \geq \kappa.
\end{equation}
Moreover this inequality is tight within a factor of $2$, \ie,
there exists $U$ with condition number $\kappa$ with
\begin{equation}\label{ineq:cond2}
2e^{-n/2} \Phi(U) \leq 2\kappa.
\end{equation}
\end{theorem}

\begin{proof}
We factor $\Phi$ into
\[
\Phi(U) = \Psi(\sigma_1,\sigma_n)
\prod_{i=2}^{n-1}\Gamma(\sigma_i),
\]
where
\[
\Psi(\sigma_1,\sigma_n) = \exp((\sigma_1^2 + \sigma_n^2)/2 )/(\sigma_1 \sigma_n),
\qquad
\Gamma(\sigma_i) = \exp(\sigma_i^2/2)/\sigma_i.
\]

We first relate $\Psi$ and the condition number, by minimizing
$\Psi(\sigma_1,\sigma_n)$ with $\sigma_1 = \kappa \sigma_n$ (\ie,
with condition number $\kappa$).
We must minimize over $\sigma_n$ the function
\[
\Psi(\kappa \sigma_n, \sigma_n)  = \frac{\exp (\sigma_n^2(1+\kappa^2)/2)}
{ \kappa \sigma_n^2}.
\]
With change of variable $z=\sigma_n^2$, this function is convex,
with minimizer $z= 2/(1+\kappa^2)$ and minimum value
$(e/2)(\kappa + 1/\kappa)$.
Therefore we have
\[
\Psi(\sigma_1,\sigma_n)  \geq (e/2)(\kappa + 1/\kappa).
\]


It is straightforward to show that $\Gamma(\sigma_i)$ is convex,
and minimized when $\sigma_i=1$.  Thus we have
$\Gamma(\sigma_i) \geq \Gamma(1) = e^{1/2}$.
We combine these results to obtain the inequality
\begin{equation}\label{ineq:cond3}
\Phi(U) \geq (e^{n/2}/2)(\kappa + 1/\kappa),
\end{equation}
which is sharp; indeed, it is tight when
\[
\sigma_1 = \left(\frac{2 \kappa^2}{1+\kappa^2}\right)^{1/2}, \qquad
\sigma_n = \left(\frac{2}{1+\kappa^2}\right)^{1/2},
\]
and $\sigma_i = 1$ for $i=2,\ldots, n-1$.

The inequality (\ref{ineq:cond3}) implies
inequality (\ref{ineq:cond}), since $\kappa + 1/\kappa \geq \kappa$.
With the values of $\sigma_i$ that make
(\ref{ineq:cond3}) tight, the inequality (\ref{ineq:cond2}) holds
because $\kappa + 1/\kappa \leq 2\kappa$.
\end{proof}

Theorems~\ref{thm:sing-vals} and~\ref{thm:cond-ineq}
show that equilibration is the same as
minimizing $\Phi(DAE)$ over diagonal $D$ and $E$, and that
$\Phi(DAE)$ is an upper bound
on $\kappa(DAE)$, the condition number of $DAE$.

\subsection{Regularized equilibration}

The equilibration problem, and its equivalent convex optimization problem
(\ref{prob:equil}), suffer from several flaws.
The first is that not all matrices can be equilibrated \cite{bradley2010algorithms}.
For example, if the nonzero matrix $A$ has a zero row or column, it cannot
be equilibrated.
As a less obvious example, a triangular matrix with unit diagonal cannot be
equilibrated.
When the matrix $A$ cannot be equilibrated,
the convex problem (\ref{prob:equil}) is unbounded \cite{fougner2015parameter}.

The second flaw is that even when the matrix $A$ can be equilibrated
problem (\ref{prob:equil}) does not have a unique solution.
Given a solution $(u^\star,v^\star)$ to problem (\ref{prob:equil}),
the point $(u^\star + \gamma, v^\star - \gamma)$ is a solution
for any $\gamma \in \reals$.
In other words, we can scale $D$ by $e^\gamma$ and $E$ by $e^{-\gamma}$
and still have $DAE$ equilibrated.
We would prefer to guarantee a solution where $D$ and $E$ have roughly
the same scale.
The final flaw is that in practice we do not want the entries of $D$
and $E$ to be extremely large or extremely small; we may have limits on
how much we are willing to scale the rows or columns.


We address these flaws by modifying the problem (\ref{prob:equil}),
adding regularization and box constraints,
and reframing the equilibration problem as
\begin{equation}\label{prob:reg-constr-equil}
\begin{array}{ll}
\mbox{minimize} & (1/2)
\sum_{i=1}^m\sum_{j=1}^{n} (A_{ij})^2e^{2u_i + 2v_j} - \alpha^2 \ones^Tu -
\beta^2 \ones^Tv + (\gamma/2)\left\|\left[\begin{array}{c}
u \\
v \end{array} \right] \right\|^2_2, \\
\mbox{subject to} & \|u\|_{\infty} \leq M, \quad
\|v\|_{\infty} \leq M,
\end{array}
\end{equation}
where $\gamma > 0$ is the regularization parameter
and the parameter $M > 0$ bounds the entries of $D$ and $E$ to lie
in the interval $[e^{-M}, e^M]$.
The additional regularization term penalizes large choices of
$u$ and $v$ (which correspond to large or small row and column scalings).
It also makes the objective strictly convex and bounded below, so
the modified problem (\ref{prob:reg-constr-equil}) always has a
unique solution $(u^\star,v^\star)$, even when $A$ cannot be equilibrated.
Assuming the constraints are not active at the solution we have
\[
\ones^Tu^\star = \ones^Tv^\star,
\]
which means that the optimal $D$ and $E$ have the same scale in the sense that
the product of their diagonal entries are equal:
\[
\prod_{i=1}^m D_{ii} = \prod_{j=1}^n E_{jj}.
\]

Problem (\ref{prob:reg-constr-equil}) is convex and can be solved
using a variety of methods.
Block alternating minimization over $u$ and $v$ can be used here,
as in the Sinkhorn-Knopp algorithm.
We minimize over $u$ (or equivalently $D$) by setting
\[
D_{ii} = \Pi_{[e^{-2M},e^{2M}]}\left(
2\alpha^2/\gamma -
W\left(2e^{2\alpha^2/\gamma}\sum_{j=1}^nA_{ij}^2E_{jj}^2 /\gamma \right)
\right)^{1/2}, \quad i=1,\ldots,m,
\]
where $W$ is the Lambert $W$ function \cite{CG:96}
and $\Pi_{[e^{-2M},e^{2M}]}$ denotes projection onto
the interval $[e^{-2M},e^{2M}]$.
We minimize over $v$ ($E$) by setting
\[
E_{jj} = \Pi_{[e^{-2M},e^{2M}]}\left(2\beta^2/\gamma -
W\left(2e^{2\beta^2/\gamma}\sum_{i=1}^mA_{ij}^2D_{ii}^2 /\gamma \right)
\right)^{1/2}, \quad j=1,\ldots,n.
\]
When $M=+\infty$, $m=n$, and $\alpha = \beta = 1$,
the $D$ and $E$ updates converge to the Sinkhorn-Knopp updates as $\gamma \to 0$
\cite{HH:08}.
This method works very well, but like the Sinkhorn-Knopp method requires
access to the individual entries of $A$, and so is not appropriate as a
matrix-free algorithm.

Of course, solving problem (\ref{prob:reg-constr-equil}) does not equilibrate
$A$ exactly; unless $\gamma =0$ and the constraints are not active,
its optimality conditions are not that $DAE$ is equilibrated.
We can make the equilibration more precise by decreasing the regularization
parameter $\gamma$ and increasing the scaling bound $M$.
But if we are using equilibration as a heuristic for reducing condition
number, approximate equilibration is more than sufficient.

\section{Stochastic method}\label{sec:sgd}

In this section we develop a method for solving
problem (\ref{prob:reg-constr-equil}) that is matrix-free, \ie, only accesses the
matrix $A$ by multiplying a vector by $A$ or by $A^T$.
(Of course we can find all the entries of $A$ by multiplying $A$ by
the unit vectors $e_i$, $i=1, \ldots, n$; then, we can use
the block minimization method described above to solve the problem.
But our hope is to solve the problem
with far fewer multiplications by $A$ or $A^T$.)

\subsection{Unbiased gradient estimate}

\paragraph{Gradient expression.}
Let $f(u,v)$ denote the objective function of problem (\ref{prob:reg-constr-equil}).
The gradient $\nabla_u f(u,v)$ is given by
\[
\nabla_{u} f(u,v) = |DAE|^2\ones - \alpha^2 + \gamma u.
\]
Similarly, the gradient $\nabla_v f(u,v)$ is given by
\[
\nabla_{v} f(u,v) = \left|EA^TD\right|^2\ones - \beta^2 + \gamma v.
\]
The first terms in these expressions,
$|DAE|^2\ones$ and $ \left|EA^TD\right|^2\ones$, are the row norms squared of the matrices
$DAE$ and $EA^TD$, respectively.  These are readily computed
if we have access to the entries of $A$; but in a matrix-free setting, where
we can only access $A$ by multiplying a vector by $A$ or $A^T$, it is difficult to
evaluate these row norms.
Instead, we will estimate them using a randomized method.

\paragraph{Estimating row norms squared.}
Given a matrix $B \in \reals^{m \times n}$,
we use the following approach to get an unbiased estimate $z$
of the row norms squared $|B|^2 \ones$.
We first sample a random vector
$s \in \reals^n$ whose entries $s_i \in \{-1,1\}$ are drawn independently
with identical distribution (IID), with probability one half for each outcome.
We then set $z = |Bs|^2$.
This technique is discussed in
\cite{bradley2010algorithms,Bekas:2007:EDM:1287840.1287969,doi:10.1080/03610919008812866}.

To see that $E[z] = |B|^2 \ones$, consider $(b^Ts)^2$,
where $b \in \reals^n$.
The expectation of $(b^Ts)^2$ is given by
\[
E[(b^Ts)^2] = \sum_{i=1}^nb_i^2E[s_i^2] + \sum_{i \neq j}b_ib_jE[s_is_j]
= \sum_{i=1}^nb_i^2.
\]
As long as the entries of $s$ are IID with mean $0$ and variance $1$,
we have $E[(b^Ts)^2] = \sum_{i=1}^nb_i^2$.
Drawing the entries of $s$ from $\{-1,1\}$, however,
minimizes the variance of $(b^Ts)^2$.


\subsection{Projected stochastic gradient}

\paragraph{Method.}
We follow the projected stochastic gradient method described in
\cite{LSB:02} and \cite[Chap.~6]{SB:15},
which solves convex optimization problems of the form
\begin{equation}\label{prob:standard-form}
\begin{array}{ll}
\mbox{minimize} & f(x) \\
\mbox{subject to} & x \in C,
\end{array}
\end{equation}
where $x \in \reals^n$ is the optimization variable,
$f : \reals^n \to \reals$ is a strongly convex differentiable function,
and $C$ is a convex set, using only an oracle that gives an unbiased estimate
of $\nabla f$, and projection onto $C$.

Specifically, we cannot evaluate $f(x)$ or $\nabla f(x)$,
but we can evaluate a function $g(x,\omega)$ and sample from a distribution $\Omega$
such that $E_{\omega \sim \Omega}g(x,\omega) = \nabla f(x)$.
Let $\mu$ be the strong convexity constant for $f$ and
$\Pi_C : \reals^n \to \reals^n$ denote the Euclidean projection onto $C$.
Then the method consists of $T$ iterations of the update
\[
\begin{split}
x^{t} &:=  \Pi_C\left(x^{t-1} - \eta_t g(x^{t-1},\omega) \right),
\end{split}
\]
where $\eta_t = 2/(\mu(t+1))$ and $\omega$ is sampled from $\Omega$.
The final approximate solution $\bar x$ is given by the weighted average
\[
\bar x = \sum_{t=0}^T\frac{2(t+1)}{(T+1)(T+2)}x^{t}.
\]
Algorithm (\ref{alg:psgd}) gives the full projected stochastic gradient
method in the context of problem (\ref{prob:reg-constr-equil}).
Recall that the objective of problem (\ref{prob:reg-constr-equil})
is strongly convex with strong convexity parameter $\gamma$.

\begin{algorithm}
\caption{Projected stochastic gradient method for problem (\ref{prob:reg-constr-equil}).}\label{alg:psgd}
\begin{algorithmic}
\Require{$u^{0}=0$, $v^{0} = 0$, $\bar{u}=0$, $\bar{v}=0$, and $\alpha, \beta, \gamma, M > 0$.}
\Statex
\For{$t=1, 2, \ldots, T$}
   \State $D \gets \diag(e^{u_1^{t-1}},\ldots, e^{u_m^{t-1}}), \qquad E \gets \diag(e^{v_1^{t-1}},\ldots, e^{v_n^{t-1}}).$
   \State Draw entries of $s \in \reals^n$ and $w \in \reals^m$ IID uniform from $\{-1,1\}$.
   \State $u^{t} \gets \Pi_{[-M,M]^m}\left(
    u^{t-1} - 2\left(
   |DAEs|^2 - \alpha^2 \ones +  \gamma u^{t-1} \right)/(\gamma(t+1)) \right).$
   \State $v^{t} \gets \Pi_{[-M,M]^n}\left(
   v^{t-1} - 2\left(
   |EA^TDw|^2 - \beta^2 \ones + \gamma v^{t-1} \right)/(\gamma(t+1)) \right).$
   \State $\bar{u} \gets 2u^{t}/(t+2) + t\bar{u}/(t+2).$
   \State $\bar{v} \gets 2v^{t}/(t+2) + t\bar{v}/(t+2).$
\EndFor
\Statex
\Ensure{$D = \diag(e^{\bar{u}_1},\ldots, e^{\bar{u}_m})$
and $E = \diag(e^{\bar{v}_1},\ldots, e^{\bar{v}_n}).$}
\end{algorithmic}
\end{algorithm}

\paragraph{Convergence rate.}
Algorithm (\ref{alg:psgd}) converges in expectation to the optimal value of
problem (\ref{prob:reg-constr-equil}) with rate $O(1/t)$ \cite{LSB:02}.
Let $f(u,v) : \reals^m \times \reals^n \to \reals$ denote the objective
of problem (\ref{prob:reg-constr-equil}),
let $(u^\star, v^\star)$ denote the problem solution,
and let $\tilde g(u,v,s,w) : \reals^m \times \reals^n \times \{-1,1\}^n \times \{-1,1\}^m \to \reals^{m+n}$ be the estimate of $\nabla f(u,v)$ given by
\[
\tilde g(u,v,s,w) = \left[\begin{array}{c}
|DAEs|^2 - \alpha^2 \ones +  \gamma u \\
|EA^TDw|^2 - \beta^2 \ones + \gamma v
\end{array}\right].
\]
Then after $T$ iterations of the algorithm we have
\[
E_{(u^T,v^T),\ldots,(u^1,v^1)} f(\bar{u}, \bar{v}) - f(u^\star, v^\star) \leq \frac{C}{\mu(T+1)},
\]
where $C$ is a constant bounded above by
\[
C \leq \max_{(u,v) \in [-M,M]^{m\times n}}2E_{s,w} \|\tilde{g}(u,v,s,w)\|^2_2.
\]
In the expectation $s$ and $w$ are random variables with entries
drawn IID uniform from $\{-1,1\}$.

We can make the bound more explicit.
It is straightforward to show the equality
\[
E_{s,w} \|\tilde{g}(u,v,s,w)\|^2_2 = \|\nabla f(u,v)\|^2_2
+ 3\ones^T\left| \left|
\left[\begin{array}{c}
DAE \\
EA^TD
\end{array}\right]
\right|^2\ones \right|^2 - 4\ones^T|DAE|^4 \ones,
\]
and the inequality
\[
\max_{(u,v) \in [-M,M]^{m\times n}} \|\nabla f(u,v)\|^2_2 \leq
\|\nabla f(M1,M1)\|^2_2 + 4\gamma M (\alpha^2 m + \beta^2 n).
\]
We combine these two results to obtain the bound
\[
C/2 \leq \|\nabla f(M1,M1)\|^2_2  + 4\gamma M (\alpha^2 m + \beta^2 n)
+ e^{8M}\left(
3\ones^T\left| \left|
\left[\begin{array}{c}
A \\
A^T
\end{array}\right]
\right|^2\ones \right|^2 - 4\ones^T|A|^4 \ones \right).
\]

Our bound on $C$ is quite large.
A more thorough analysis could improve the bound by considering the relative sizes of the different
parameters and entries of $A$.
For instance, it is straightforward to show that for $t=1,\ldots,T$ we have
\[
u^t_i \leq \alpha^2/\gamma, \quad i=1,\ldots,m, \qquad
v^t_j \leq \beta^2/\gamma, \quad j=1,\ldots,n,
\]
which gives a tighter bound if $\alpha^2/\gamma < M$ or $\beta^2/\gamma < M$.
In any case, we find that in practice no more than tens
of iterations are required to reach an approximate solution.



\section{Numerical experiments}\label{sec:num-exp}
We evaluated algorithm (\ref{alg:psgd}) on many different matrices $A$.
We only describe the results for a single numerical experiment,
but we obtained similar results for our other experiments.
For our numerical experiment we generated a sparse matrix
$\hat{A} \in \reals^{m \times n}$,
with $m = 2\times 10^4$ and $n=10^4$,
with 1\% of the entries chosen uniformly at random to be nonzero,
and nonzero entries drawn IID from a standard normal distribution.
We next generated vectors $\hat{u} \in \reals^m$ and
$\hat{v} \in \reals^n$ with entries drawn IID from a normal distribution with
mean 1 and variance 1.
We set the final matrix $A$ to be
\[
A = \diag\left(e^{\hat{u}_1},\ldots,e^{\hat{u}_m}\right)
\hat{A}
\diag\left(e^{\hat{v}_1},\ldots,e^{\hat{v}_n}\right).
\]

We ran algorithm (\ref{alg:psgd}) for 1000 iteratons to obtain
an approximate solution $f(\bar{u},\bar{v})$.
We used the parameters $\alpha = (\frac{n}{m})^{1/4}$,
$\beta = (\frac{m}{n})^{1/4}$, $\gamma = 10^{-1}$ and $M = \log(10^4)$.
We obtained the exact solution $p^\star$ to high accuracy
using Newton's method with backtracking line search.
(Newton's method does not account for constraints,
but we verified that the constraints were not active at the solution.)

Figure \ref{fig:equil} plots the relative optimality gap
$(f(\bar{u},\bar{v}) - p^\star)/f(0,0)$ and the RMS equilibration error,
\[
\frac{1}{\sqrt{m+n}}\left( \sum_{i=1}^m\left( \sqrt{e_i^T|DAE|^2\ones} - \alpha\right)^2
+ \sum_{j=1}^n\left( \sqrt{e_j^T\left|EA^TD \right|^2\ones} - \beta\right)^2 \right)^{1/2},
\]
versus iteration.
The RMS error shows how close $DAE$ is to equilibrated;
we do not expect it to converge to zero because of the regularization.

The objective value and RMS error decrease quickly
for the first few iterations, with oscillations,
and then decrease smoothly but more slowly.
The slopes of the lines show the convergence rate.
The least-squares linear fit for the optimality gap has slope $-2.0$,
which indicates that the convergence was (much) faster than the theoretical
upper bound $1/t$.

\begin{figure}
\begin{center}
\includegraphics[width=0.7\textwidth]{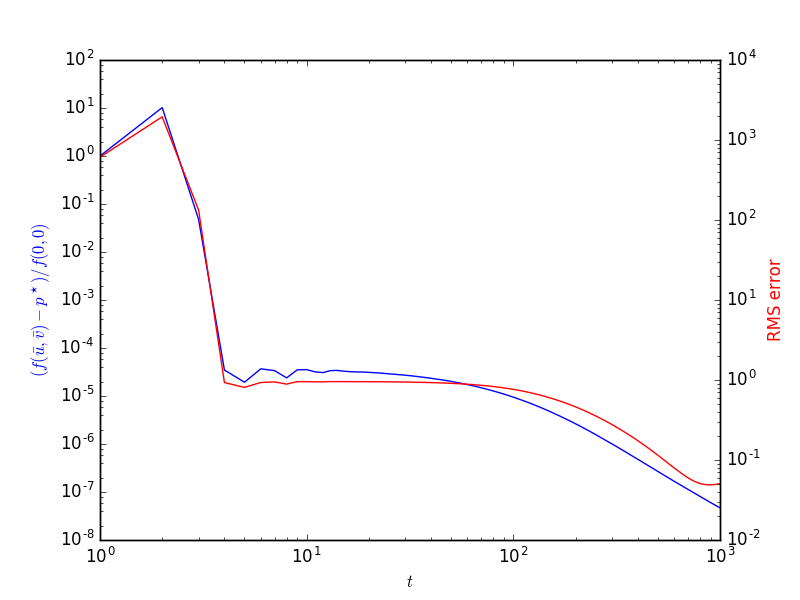}
\end{center}
\caption{Problem (\ref{prob:reg-constr-equil}) optimality gap and RMS error versus iterations $t$.
}\label{fig:equil}
\end{figure}

Figure \ref{fig:cond} shows the condition number of $DAE$ versus iteration.
While equilibration merely minimizes an upper bound on the condition number,
in this case the condition number corresponded quite closely with
the objective of problem (\ref{prob:reg-constr-equil}).
The plot shows that after 4 iterations $\kappa(DAE)$ is back to the
original condition number $\kappa(A) = 10^4$.
After 100 iterations the condition number is reduced by $200 \times$,
and it continues to decrease with further iterations.

\begin{figure}
\begin{center}
\includegraphics[width=0.7\textwidth]{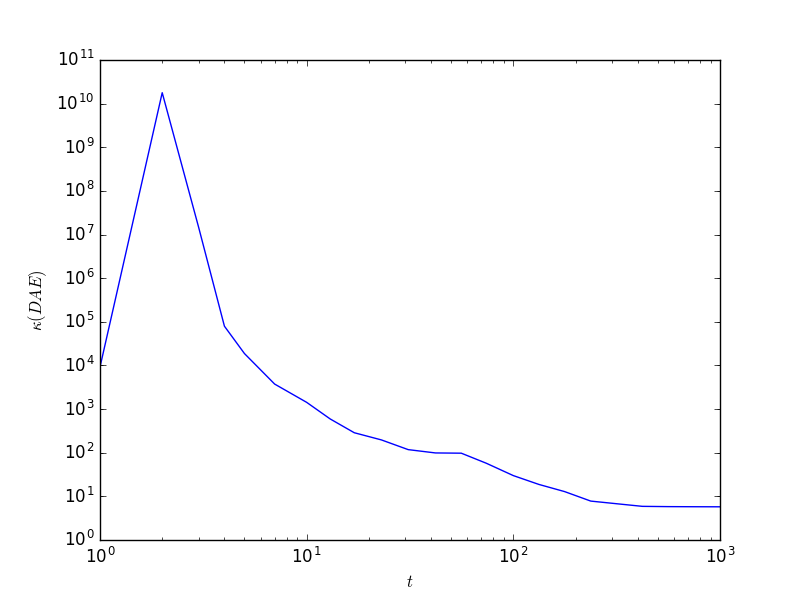}
\end{center}
\caption{Condition number of $DAE$ versus iterations $t$.
}\label{fig:cond}
\end{figure}

\section{Applications}

\subsection{LSQR}
The LSQR algorithm \cite{PaS:82} is an iterative matrix-free method for solving the
linear system $Ax = b$,
where $x\in \reals^n$, $A \in \reals^{n \times n}$
and $b \in \reals^n$.
Each iteration of LSQR involves one multiplication by $A$ and one by $A^T$.
LSQR is equivalent in exact arithmetic to applying the conjugate
gradients method \cite{HeS:52} to the normal equations $A^TAx = A^Tb$,
but in practice has better numerical properties.
An upper bound on the number of iterations needed by LSQR to
achieve a given accuracy grows with $\kappa(A)$ \cite[Chap.~5]{NW:06}.
Thus decreasing the condition number of $A$ via equilibration
can accelerate the convergence of LSQR.
(Since LSQR is equivalent to conjugate gradients applied to the normal equations,
it computes the exact solution in $n$ iterations, at least in exact arithmetic.
But with numerical roundoff error this does not occur.)

We use equilibration as a preconditioner by solving the
linear system $(DAE)\bar{x} = Db$ with LSQR instead of $Ax = b$; we then
recover $x$ from $\bar{x}$ via $x = E \bar{x}$.
We measure the accuracy of an approximate solution $\bar{x}$ by the
residual $\|Ax - b\|_2$
rather than by residual $\|DAE\bar{x} - Db\|_2$ of the preconditioned system,
since our goal is to solve the original system $Ax = b$.

We compared the convergence rate of LSQR with and
without equilibration.
We generated the matrix $\hat{A} \in \reals^{n \times n}$ as in
\S\ref{sec:num-exp}, with $n=10^4$.
We choose $b \in \reals^n$ by first generating $x^\star \in \reals^n$ by drawing entries IID
from a standard normal distribution, and then
setting $b = Ax^\star$.

We generated equilibrated matrices $D_{10}AE_{10}$, $D_{30}AE_{30}$,
$D_{100}AE_{100}$, and $D_{300}AE_{300}$
by running algorithm (\ref{alg:psgd}) for 10, 30, 100, and 300 iterations,
respectively.
We used the parameters $\alpha = (n/m)^{1/4}$,
$\beta = (m/n)^{1/4}$, $\gamma = 10^{-1}$ and $M = \log(10^4)$.
Note that the cost of equlibration iterations is the same as the cost
of LSQR iterations, since each involves one multiply by $A$ and one by $A^T$.

Figure \ref{fig:lsqr} shows the results of running LSQR with and without equilibration,
from the initial iterate $x^{0} = 0$.  We show the relative residual
$\|Ax^{t} - b\|_2/\|b\|_2$ versus iterations,
counting the equilibration iterations,
which can be seen as the original flat portions at the beginning of each curve.
We can see that to achieve relative accuracy $10^{-4}$, LSQR without
preconditioning requires around $10^4$ iterations; with preconditioning
with 30 or more iterations of equilibration, it requires more than
$10\times $ fewer iterations.
We can see that higher accuracy justifies more equilibration iterations, but
that the choice of just $30$ equilibration iterations does very well.
We can see that 10 iterations of equilibration is too few,
and only improves LSQR convergence a small amount.

\begin{figure}
\begin{center}
\includegraphics[width=0.7\textwidth]{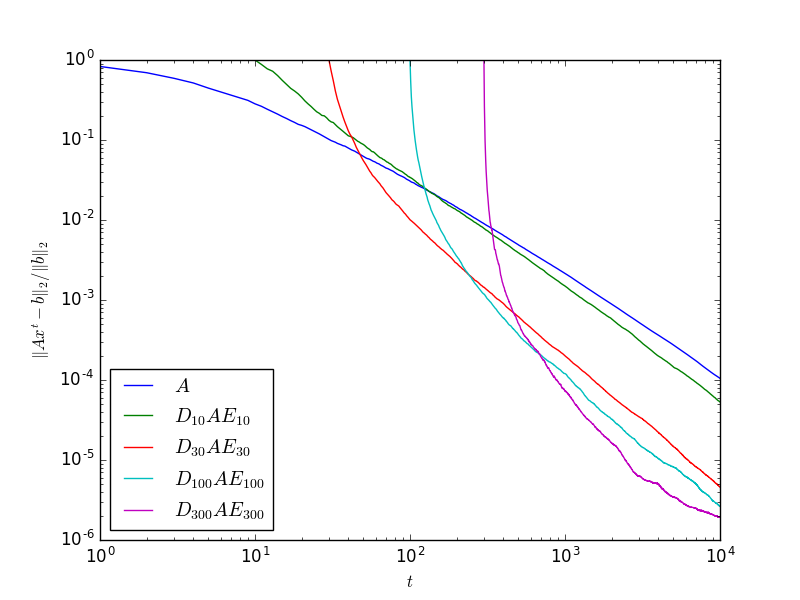}
\end{center}
\caption{Residual versus iterations $t$ for LSQR.
}\label{fig:lsqr}
\end{figure}

\subsection{Chambolle-Cremers-Pock}
The Chambolle-Cremers-Pock (CCP) algorithm \cite{CP:11,PCBC:09}
is an iterative method for solving convex optimization problems of the
form
\[
\begin{array}{ll}
\mbox{minimize} & f(x) + g(Ax),
\end{array}
\]
where $x \in \reals^n$ is the variable,
$A \in \reals^{m \times n}$ is problem data,
and $f$ and $g$ are convex functions.
Each iteration of CCP requires one multiplication by $A$ and one by $A^T$.
Chambolle and Pock do not show a dependence on $\kappa(A)$ in their
analysis of the algorithm convergence rate,
but we nonetheless might expect that equilibration will
accelerate convergence.

We compared the convergence rate of CCP with and
without equilibration on the Lasso problem \cite[\S 3.4]{friedman2001elements}
\[
\begin{array}{ll}
\mbox{minimize} & \|Ax - b\|^2_2/\sqrt{\lambda} + \sqrt{\lambda}\|x\|_1.
\end{array}
\]
We generated the matrix $A \in \reals^{m \times n}$ as in
\S\ref{sec:num-exp}, with $m=10^4$ and $n=2 \times 10^4$.
We generated $b \in \reals^m$ by first generating $\hat{x} \in \reals^n$
by choosing $n/10$ entries uniformly at random to be nonzero
and drawing those entries IID from a standard normal distribution.
We then set $b = A\hat{x} + \nu$,
where the entries of $\nu \in \reals^m$ were drawn IID from
a standard normal distribution.
We set $\lambda = 10^{-3} \|A^Tb\|_\infty$ and found the optimal value $p^\star$
for the Lasso problem using CVXPY \cite{cvxpy_paper} and GUROBI \cite{gurobi}.

We generated equilibrated matrices $D_{10}AE_{10}$, $D_{30}AE_{30}$,
$D_{100}AE_{100}$, and $D_{300}AE_{300}$
by running algorithm (\ref{alg:psgd}) for 10, 30, 100, and 300 iterations,
respectively.
We used the parameters $\alpha = (n/m)^{1/4}$,
$\beta = (m/n)^{1/4}$, $\gamma = 10^{-1}$ and $M = \log(10^4)$.

Figure \ref{fig:ccp} shows the results of running CCP with and without equilibration.
We used the parameters $\tau = \sigma = 0.9/\|D_kAE_k\|_2$ and $\theta = 1$
and set all initial iterates to 0.
We show the relative optimality gap $(f(x^{t}) - p^\star)/f(0)$ versus iterations,
counting the equilibration iterations,
which can be seen as the original flat portions at the beginning of each curve.
We can see that to achieve relative accuracy $10^{-6}$, CCP without
preconditioning requires around $1000$ iterations; with preconditioning
with 100 iterations of equilibration, it requires more than
$4\times $ fewer iterations.
CCP converges to a highly accurate solution with
just 100 equilibration iterations,
so additional equilibration iterations are unnecessary.
We can see that 10 and 30 iterations of equilibration are too few,
and do not improve CCP's convergence.

\begin{figure}
\begin{center}
\includegraphics[width=0.7\textwidth]{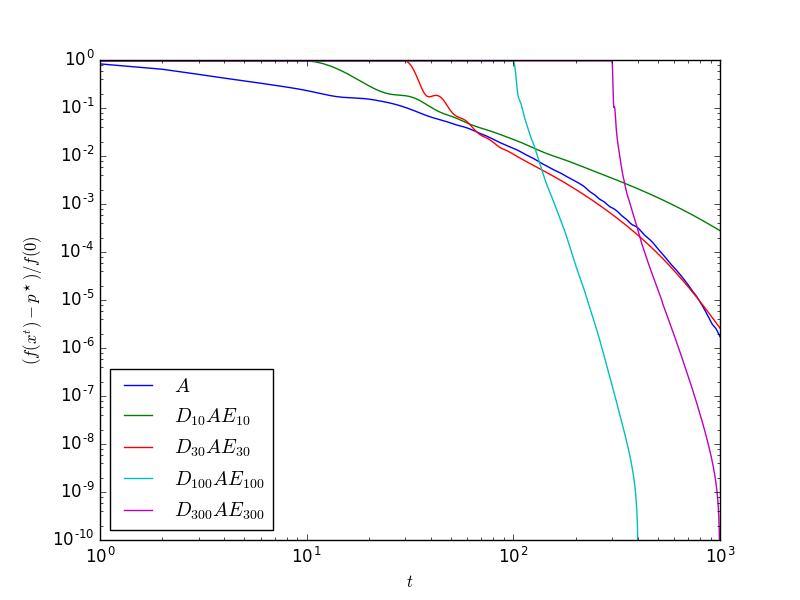}
\end{center}
\caption{Optimality gap versus iterations $t$ for CCP.
}\label{fig:ccp}
\end{figure}

\section{Variants}

In this section we discuss several variants of the equilibration
problem that can also be solved in a matrix-free manner.

\paragraph{Symmetric equilibration.}
When equilibrating a symmetric matrix $A \in \reals^{n \times n}$,
we often want the equilibrated matrix $DAE$ to also be symmetric.
For example, to use equilibration as a preconditioner for
the conjugate gradients method, $DAE$ must be symmetric \cite{HeS:52}.
We make $DAE$ symmetric by setting $D=E$.

Symmetric equilibration can be posed as the convex
optimization problem
\begin{equation}\label{prob:sym-equil}
\begin{array}{ll}
\mbox{minimize} & (1/4)
\sum_{i=1}^n\sum_{j=1}^{n} (A_{ij})^2e^{2u_i + 2u_j} - \alpha^2 \ones^Tu,
\end{array}
\end{equation}
where $u \in \reals^n$ is the optimization variable and $\alpha > 0$
is the desired value of the row and column norms.
We approximately solve problem (\ref{prob:sym-equil}) by adding
regularization and box constraints as in problem
(\ref{prob:reg-constr-equil})
and then applying algorithm (\ref{alg:sym-psgd}),
a simple modification of algorithm (\ref{alg:psgd})
with the same convergence guarantees.

\begin{algorithm}
\caption{Projected stochastic gradient method for symmetric equilibration.}\label{alg:sym-psgd}
\begin{algorithmic}
\Require{$u^{0}=0$, $\bar{u}=0$, and $\alpha, \gamma, M > 0$.}
\Statex
\For{$t=1, 2, \ldots, T$}
   \State $D \gets \diag(e^{u_1^{t-1}},\ldots, e^{u_n^{t-1}}).$
   \State Draw entries of $s \in \reals^n$ IID uniform from $\{-1,1\}$.
   \State $u^{t} \gets \Pi_{[-M,M]^n}\left(
    u^{t-1} - 2\left(
   |DADs|^2 - \alpha^2 1 +  \gamma u^{t-1} \right)/(\gamma(t+1)) \right).$
   \State $\bar{u} \gets 2u^{t}/(t+2) + t\bar{u}/(t+2).$
\EndFor
\Statex
\Ensure{$D = \diag(e^{\bar{u}_1},\ldots, e^{\bar{u}_n})$.}
\end{algorithmic}
\end{algorithm}

\paragraph{Varying row and column norms.}
In standard equilibration we want all the row norms of $DAE$ to be
the same and all the column norms to be the same.
We might instead want the row and column norms to equal known
vectors $r \in \reals^m$ and $c \in \reals^n$, respectively.
The vectors must satisfy $r^Tr = c^Tc$.

Equilibration with varying row and column norms can be posed as
the convex optimization problem
\begin{equation}\label{prob:equil-vary-norms}
\begin{array}{ll}
\mbox{minimize} & (1/2)
\sum_{i=1}^m\sum_{j=1}^{n} (A_{ij})^2e^{2u_i + 2v_j} - r^Tu - c^Tv,
\end{array}
\end{equation}
where as usual $u \in \reals^m$ and $v \in \reals^n$ are the optimization variables.
We approximately solve problem (\ref{prob:equil-vary-norms}) by adding
regularization and box constraints as in problem
(\ref{prob:reg-constr-equil})
and then applying algorithm (\ref{alg:psgd}) with the appropriate
modification to the gradient estimate.

\paragraph{Block equilibration.}
A common constraint when using equilibration as a preconditioner
is that the diagonal entries of $D$ and $E$ are divided into blocks
that all must have the same value.
For example, suppose we have a cone program
\[
\begin{array}{ll}
\mbox{minimize}   & c^Tx \\
\mbox{subject to} & Ax + b \in \K,
\end{array}
\]
where $x \in \reals^n$ is the optimization variable,
$c \in \reals^n$, $b \in \reals^m$, and $A \in \reals^{m\times n}$
are problem data,
and $\K = \K_1 \times \cdots \times \K_\ell$ is a product of
convex cones.

If we equilibrate $A$ we must ensure that $D\K = \K$.
Let $m_i$ be the dimension of cone $\K_i$.
A simple sufficient condition for $D\K = \K$ is that $D$ have the form
\begin{equation}\label{eq:D-form}
D = \diag(e^{u_1} I_{m_1}, \ldots, e^{u_p}I_{m_p}),
\end{equation}
where $u \in \reals^p$ and $I_{m_i}$ is the $m_i$-by-$m_i$ identity matrix.
Given the constraint on $D$, we cannot ensure that each row of $DAE$
has norm $\alpha$.
Instead we view each block of $m_i$ rows as a single vector
and require that the vector have norm $\sqrt{m_i}\alpha$.

In the full block equilibration problem we also require that $E$ have the form
\begin{equation}\label{eq:E-form}
E = \diag(e^{v_1} I_{n_1}, \ldots, e^{v_q}I_{n_q}),
\end{equation}
where $v \in \reals^q$ and $I_{n_j}$ is the $n_j$-by-$n_j$ identity matrix.
Again, we view each block of $n_j$ columns as a single vector
and require that the vector have norm $\sqrt{n_j}\beta$.

Block equilibration can be posed as the convex optimization problem
\begin{equation}\label{prob:block-equil}
\begin{array}{ll}
\mbox{minimize} & (1/2)
\ones^T|DAE|^2 \ones
- \alpha^2u^T\left[\begin{array}{c}
m_1 \\
\vdots \\
m_p \end{array}\right]
- \beta^2v^T\left[\begin{array}{c}
n_1 \\
\vdots \\
n_q \end{array}\right],
\end{array}
\end{equation}
where $D$ and $E$ are defined as in equations (\ref{eq:D-form}) and
(\ref{eq:E-form}).
We approximately solve problem (\ref{prob:block-equil}) by adding
regularization and box constraints as in problem
(\ref{prob:reg-constr-equil})
and then applying algorithm (\ref{alg:psgd}) with the appropriate
modification to the gradient estimate.
Our stochastic matrix-free block equilibration method is used in
the matrix-free versions of the cone solvers SCS \cite{SCSpaper} and
POGS \cite{fougner2015parameter} described in \cite{DB:ICCV,DB:15}.

\paragraph{Tensor equilibration.} We describe here the case of 3-tensors;
the generalization to higher order tensors is clear.
We are given a $3$-dimensional array $A \in \reals^{m \times n \times p}$,
and seek coordinate scalings
$d \in \reals^{m}, e \in \reals^{n}, f \in \reals^{p}$
for which
\[
\begin{array}{cc}
\left( \sum_{j=1}^{n}\sum_{k=1}^p
A_{ijk}^2 d_i^2 e_j^2 f_k^2 \right)^{1/2}  = \alpha,
\quad &i=1,\ldots,m \\
\left( \sum_{i=1}^{m}\sum_{k=1}^p
A_{ijk}^2 d_i^2 e_j^2 f_k^2 \right)^{1/2}  = \beta,
\quad &j=1,\ldots,n \\
\left( \sum_{i=1}^{m}\sum_{j=1}^{n}
A_{ijk}^2 d_i^2 e_j^2 f_k^2 \right)^{1/2}  = \gamma,
\quad &k=1,\ldots,p.
\end{array}
\]
Here $\alpha, \beta, \gamma > 0$ are constants that satisfy
$m\alpha^2=n\beta^2= p \gamma^2$.

Tensor equilibration can be posed as the convex optimization problem
\begin{equation}\label{prob:tensor}
\begin{array}{ll}
\mbox{minimize} & (1/2)
\sum_{i=1}^m\sum_{j=1}^{n}\sum_{k=1}^p
(A_{ijk}^2)e^{2(u_i + v_j + w_k)} -
\alpha^2\ones^Tu - \beta^2 \ones^Tv - \gamma^2 \ones^Tw,
\end{array}
\end{equation}
where $u \in \reals^m$, $v \in \reals^n$, and $w \in \reals^p$ are the optimization variables.
We can solve problem (\ref{prob:tensor}) using a simple variant of algorithm (\ref{alg:psgd})
that only interacts with the array $A$ via the matrix-to-vector operations
\[
\begin{array}{ll}
X &\to \sum_{i=1}^{m}\sum_{j=1}^{n}
A_{ijk}X_{ij} \\
Y &\to \sum_{i=1}^{m}\sum_{k=1}^p
A_{ijk}Y_{ik} \\
Z &\to \sum_{j=1}^{n}\sum_{k=1}^p
A_{ijk}Z_{jk}.
\end{array}
\]

\section*{Acknowledgments}
The authors thank Reza Takapoui for helpful comments and pointers.
This material is based upon work supported by the
National Science Foundation Graduate
Research Fellowship under Grant No. DGE-114747
and by the DARPA X-DATA and SIMPLEX programs.

\newpage
\bibliography{mf_equil}

\end{document}